\documentclass[12pt]{amsart}
\usepackage{amscd}
\usepackage{amsmath}
\usepackage{stmaryrd}
\usepackage{amssymb}
\usepackage{units}
\usepackage[all]{xy}
\def\frk{\frak}               

\def\Phi{{\frk n}}
\def\Phi{{\frk N}}
\def\opn#1#2{\def#1{\operatorname{#2}}} 
\opn\chara{char} \opn\length{\ell} \opn\pd{pd} \opn\rk{rk}
\opn\projdim{proj\,dim} \opn\injdim{inj\,dim} \opn\rank{rank}
\opn\depth{depth} \opn\sdepth{sdepth} \opn\fdepth{fdepth}
\opn\grade{grade} \opn\height{height} \opn\embdim{emb\,dim}
\opn\codim{codim}  \opn\min{min} \opn\max{max}

\opn\Tr{Tr} \opn\bigrank{big\,rank}
\opn\superheight{superheight}\opn\lcm{lcm}
\opn\trdeg{tr\,deg}
\opn\reg{reg} \opn\lreg{lreg} \opn\ini{in} \opn\lpd{lpd}
\opn\size{size}
\opn\div{div} \opn\Div{Div} \opn\cl{cl} \opn\Cl{Cl}
\opn\Spec{Spec} \opn\Supp{Supp} \opn\supp{supp} \opn\Sing{Sing}
\opn\Ass{Ass} \opn\Min{Min}
\opn\Ann{Ann} \opn\Rad{Rad} \opn\Soc{Soc}
\opn\Im{Im} \opn\Ker{Ker} \opn\Coker{Coker} \opn\Am{Am}
\opn\Hom{Hom} \opn\Tor{Tor} \opn\Ext{Ext} \opn\End{End}
\opn\Aut{Aut} \opn\id{id}  \opn\deg{deg}

\opn\nat{nat}
\opn\pff{pf}
\opn\Pf{Pf} \opn\GL{GL} \opn\SL{SL} \opn\mod{mod} \opn\ord{ord}
\opn\Gin{Gin} \opn\Hilb{Hilb}
\opn\aff{aff} \opn\con{conv} \opn\relint{relint} \opn\st{st}
\opn\lk{lk} \opn\cn{cn} \opn\core{core} \opn\vol{vol}
\opn\link{link} \opn\star{star}
\opn\gr{gr}

\def\pot#1#2{#1[\kern-0.28ex[#2]\kern-0.28ex]}

\newcommand{\fracs}[2]{\displaystyle\frac{#1}{#2}}
\newcommand{\power}[1]{\llbracket #1 \rrbracket}
\newcommand{\Sum}[2]{\displaystyle\sum_{#1}^{#2}}

\def\sing{\textsc{Singular}}
\def\ti{\scriptstyle}
\opn\dirlim{\underrightarrow{\lim}}
\opn\inivlim{\underleftarrow{\lim}}
\let\to=\rightarrow

\def\Implies{\ifmmode\Longrightarrow \else
        \unskip${}\Longrightarrow{}$\ignorespaces\fi}
\def\implies{\ifmmode\Rightarrow \else
        \unskip${}\Rightarrow{}$\ignorespaces\fi}
\def\iff{\ifmmode\Longleftrightarrow \else
        \unskip${}\Longleftrightarrow{}$\ignorespaces\fi}

\let\:=\colon
\newtheorem{Theorem}{Theorem}[]
\newtheorem{Lemma}[Theorem]{Lemma}
\newtheorem{Corollary}[Theorem]{Corollary}

\theoremstyle{definition}
\newtheorem{Remark}[Theorem]{Remark}

\newtheorem{Example}[Theorem]{Example}

\newtheorem{Definition}[Theorem]{Definition}

\let\epsilon\varepsilon
\let\phi=\varphi
\let\kappa=\varkappa
\def\qed{\ifhmode\textqed\fi
      \ifmmode\ifinner\quad\qedsymbol\else\dispqed\fi\fi}
\def\textqed{\unskip\nobreak\penalty50
       \hskip2em\hbox{}\nobreak\hfil\qedsymbol
       \parfillskip=0pt \finalhyphendemerits=0}
\def\dispqed{\rlap{\qquad\qedsymbol}}

\opn\dis{dis}
\def\pnt{{\raise0.5mm\hbox{\large\bf.}}}

\opn\Lex{Lex}

\begin{document}
\title{\bf A method to compute the General Neron Desingularization in the frame of one dimensional local domains}
\author{ Adrian Popescu, Dorin Popescu }
\thanks{The  support from the Department of Mathematics of the University of Kaiserslautern of the first author and the support from the project  ID-PCE-2011-1023, granted by the Romanian National Authority for Scientific Research, CNCS - UEFISCDI  of the second author are gratefully acknowledged. Both authors thank  CIRM, Luminy who provided excellent conditions and stimulative atmosphere in the main stage of our work.}

\address{Adrian Popescu,  Department of Mathematics, University of Kaiserslautern, Erwin-Schr\"odinger-Str., 67663 Kaiserslautern, Germany}
\email{popescu@mathematik.uni-kl.de}

\address{Dorin Popescu, Simion Stoilow Institute of Mathematics , Research unit 5,
University of Bucharest, P.O.Box 1-764, Bucharest 014700, Romania}
\email{dorin.popescu@imar.ro}

\maketitle

\begin{abstract} An algorithmic proof of General Neron Desingularization is given here for one dimensional local domains and it is implemented in \textsc{Singular}. Also a theorem recalling Greenberg' strong approximation theorem is presented for one dimensional Cohen-Macaulay local rings.

\end{abstract}

\section*{Introduction}

A ring morphism $u:A\to A'$ has  {\em regular fibers} if for all prime ideals $P\in \Spec A$ the ring $A'/PA'$ is a regular  ring, i.e. its localizations are regular local rings. It has {\em geometrically regular fibers}  if for all prime ideals $P\in \Spec A$ and all finite field extensions $K$ of the fraction field of $A/P$ the ring  $K\otimes_{A/P} A'/PA'$ is regular.
If for all $P\in \Spec A$ the fraction field of $A/P$ has characteristic $0$ then the regular fibers of $u$ are geometrically regular fibers. A flat morphism $u$ is {\em regular} if its fibers are geometrically regular.

In Artin approximation theory \cite{A} an important result is the following theorem generalizing the Neron Desingularization \cite{N}, \cite{A}.

\begin{Theorem}[General Neron Desingularization, Popescu \cite{P0}, \cite{P}, \cite{P1},  Andre \cite{An}, Swan \cite{S}, Spivakovski \cite{Sp}]\label{gnd}  Let $u:A\to A'$ be a  regular morphism of Noetherian rings and $B$ a finite type $A$-algebra. Then  any $A$-morphism $v:B\to A'$   factors through a smooth $A$-algebra $C$, that is $v$ is a composite $A$-morphism $B\to C\to A'$.
\end{Theorem}

The purpose of this paper is to give an algorithmic proof of the above theorem when $A,A'$ are one dimensional local domains and $A\supset \mathbb Q$. This proof is somehow presented by the second author in a lecture given in the frame of a semester of Artin Approximation and Singularity Theory organized in 2015 by CIRM in Luminy (see http://hlombardi.free.fr/Popescu-Luminy2015.pdf). The algorithm was implemented by the authors in the Computer Algebra system \textsc{Singular} \cite{Sing} and will be as soon as possible  found in a development version as the library \verb"GND.lib" at \begin{verbatim}https://github.com/Singular/Source.\end{verbatim}

We may take the same General Neron Desingularization for $v,v':B\to A'$ if they are closed enough as Examples \ref{ex 4}, \ref{exc} show. The last section computes the General Neron Desingularization in several examples. We should point that the General Neron Desingularization is not unique and it is better to speak above about {\bf a} General Neron Desingularization.

When $A'$ is the completion of a Cohen-Macaulay local ring $A$ of dimension one we show that we may have a linear Artin function as it happens in the Greenberg's case (see \cite{Gr}). More precisely, the Artin function is given by $c\to 2e+c$, where $e$ depends from the polynomial system of equations defining $B$ (see Theorem \ref{gr}).

\vskip 0.5 cm

\section{The theorem}

 Let $u:A\to A'$ be a flat morphism of Noetherian local domains of dimension $1$. Suppose that $A\supset \mathbb Q$  and  the maximal ideal $m$ of $A$ generates the maximal ideal of $A'$. Then $u$ is regular morphism. Moreover, we suppose that there exist canonical inclusions $k=A/m\to A$, $k'=A'/mA'\to A'$ such that $u(k)\subset k'$.

  Let $B=A[Y]/I$, $Y=(Y_1,\ldots,Y_n)$. If $f=(f_1,\ldots,f_r)$, $r\leq n$ is a system of polynomials from $I$ then we can define the ideal $\Delta_f$ generated by all $r\times r$-minors of the Jacobian matrix $\left(\fracs{\partial f_i}{\partial Y_j}\right)$. After Elkik \cite{El} let $H_{B/A}$ be the radical of the ideal $\sum_f \big((f):I\big)\Delta_fB$, where the sum is taken over all systems of polynomials $f$ from $I$ with $r\leq n$.
Then $B_P$, $P\in \Spec B$ is essentially smooth over $A$ if and only if $P\not \supset H_{B/A}$ by the Jacobian criterion for smoothness.
   Thus  $H_{B/A}$ measures the non smooth locus of $B$ over $A$.

\begin{Definition}
  $B$ is {\bf standard smooth} over $A$ if  there exists  $f$ in $I$ as above such that $1\in \big((f):I\big)\Delta_fB$.
\end{Definition}
  The aim of this paper is to give an easy algorithmic  proof of the following theorem.

 \begin{Theorem} \label{m} Any $A$-morphism $v:B\to A'$   factors through a standard smooth $A$-algebra $B'$.
 \end{Theorem}

If $A$ is  essentially of finite type over $\mathbb Q$, then the ideal $H_{B/A}$ can be  computed in \textsc{Singular} by following its definition but it is easier to describe only the ideal $\sum_f \big((f):I\big)\Delta_fB$ defined above.  This is the case  considered in our algorithmic part, let us say $A\cong k[x]/F$ for some variables $x=(x_1,\ldots x_t)$, and the completion of $A'$ is $K\llbracket x\rrbracket/(F)$ for some  field extension $k\subset K$. When $v$ is defined by polynomials $y$ from $K[x]$ then our problem is easy. Let $L$ be the field obtained by adjoining to $k$ all coefficients of $y$. Then $R=L[x]/(F)$ is a subring of $A'$ containing $\Im v$ which is essentially smooth over $A$. Then we may take $B' $ as a standard smooth $A$-algebra such that $R$ is a localization of $B'$.
Consequently we  suppose usually  that $y$ is not polynomial defined and moreover $L$ is not a finite type field extension of $k$.

\section{Reduction to the case when $H_{B/A}\cap A\not =0$.}

We may suppose that $v(H_{B/A})\not =0$. Indeed, if $v(H_{B/A}) =0$ then $v$ induces an $A$-morphism $v':B'=B/H_{B/A}\to A'$ and we may change $(B,v)$ by $(B',v')$. Applying this trick several times we reduce to the case  $v(H_{B/A})\not =0$. However the fraction field of $\Im v$ is essentially smooth over $A$ by separability, that is $H_{\Im v/A}A'\not =0$ and in the worst case our trick will change $B$ by $\Im v$ after several steps.

  Choose $P'\in \big(\Delta_f((f):I)\big)\setminus I$ for some system of polynomials $f=(f_1,\ldots,f_r)$ from $I$ and  $d'\in \big(v(P')A'\big )\cap A$, $d'\not = 0$. Moreover we may choose $P'$ to be from $M\big((f):I\big)$ where $M$ is a $r\times r$-minor of $\left(\fracs{\partial f}{\partial Y}\right)$ . Then $d'=v(P')z\in \big(v(H_{B/A})\big)\cap A$ for some $z\in A'$. Set $B_1=B[Z]/(f_{r+1}) $, where $f_{r+1}=-d'+P'Z$
and let $v_1:B_1\to A'$ be the map of $B$-algebras given by $Z\to z$. It follows that $d'\in \big(( f,f_{r+1}):(I,f_{r+1})\big)$ and  $d'\in \Delta_f$, $d'\in \Delta_{f_{r+1}}$. Then $d=d'^2\equiv P\ \mbox{modulo}\ (I,f_{r+1})$ for $P=P'^2Z^2\in H_{B_1/A}$. For the reduction change $B$ by $B_1$ and the Jacobian matrix $J=(\partial f/\partial Y)$ will be  now the new $J$ given by $
\left(\begin{array}{cc}
J & 0 \\
* & P'
\end{array}\right).$
Note that $d\in H_{B/A}\cap A$.

\begin{Example}\label{ex 4} Let  $a_1,a_2\in \mathbb C$  be two elements algebraically independent over $\mathbb Q$ and  $\rho$  a root of the polynomial $T^2+T+1$ in $ \mathbb C$. Then $k'=\fracs{\mathbb Q(a_1, a_2)[a_3]}{(a_3^2+a_3+1)}\cong \mathbb Q(\rho,a_1,a_2) $. Let
   $A = \left(\fracs{\mathbb Q[x_1, x_2]}{(x_1^3-x_2^2)}\right)_{(x_1,x_2)}$ and $B = \fracs{A[Y_1,Y_2,Y_3]}{(Y_1^3-Y_2^3)}$,  $A' = \fracs{k'\power{x_1,x_2}}{(x_1^3-x_2^2)}$ and the map $v$ defined as
$$\begin{xy}
\xymatrix@R-2pc{
v: & B \ar[r]& A' \\
&Y_1 \ar@{|->}[r] & a_1x_2\\
&Y_2 \ar@{|->}[r] & a_1a_3x_2\\
&Y_3 \ar@{|->}[r] & a_1 \Sum{i=0}{30}\fracs{x_1^i}{i!} + a_2x_2\Sum{i=31}{50}\fracs{x_1^i}{i!}
}
\end{xy}$$

This is an easy example. Indeed, let $v'':B''=A[a_3,a_1x_2,v(Y_3)]\to A'$ be the inclusion. We have $\Im v\subset B''\cong \fracs{A[T,Y_1,Y_3]}{(T^2+T+1)}$ and $B''_{2a_3+1}\cong \left(\fracs{A[T,Y_1,Y_3]}{(T^2+T+1)}\right)_{2T+1}$ is a smooth $A$-algebra, which  could be taken as a General Neron Desingularization of $B$.
Applying our algorithm we will get more complicated General Neron Desingularizations but useful for an illustration of our construction.

Then $\Im v$, the new $B$ will be $\fracs{B}{\Ker v}$, where the kernel is generated by the following polynomial:
\begin{verbatim}
ker[1]=Y1^2+Y1*Y2+Y2^2
\end{verbatim}

Next we choose $f = Y_1^2+Y_1Y_2+Y_2^2$ and we have $M = 2Y_2+Y_1$ and $1 \in \big((f):I\big)$ and hence $P' = Y_1 + 2Y_2$. Therefore $v(P') = (2a_1a_3 + a_1) \cdot x_2$ and $d' = x_2$, $z = \fracs{1}{2a_1a_3+a_1}$. Therefore $d = d'^2 = x_2^2$.

To be able to construct $\mathbb Q\left[\fracs{1}{2a_1a_3+a_1}\right][x]$ in $\sing$ we will  add a new variable  $a$ and we will  factorize with the corresponding polynomial $2a_1a_3 \cdot a + a_1 \cdot a -1$. We will see this $a$ as a new parameter from $k'\subset A'$. Then we change $B$ by $B_1=\fracs{B[Y_4]}{(-d'+P'Y_4)}$ and extend $v$ to a map $v_1:B_1\to A'$ given by $Y_4\to a$.  Changing $B$ by $B_1$ we may assume that $d\in H_{B/A}\cap A$.

\end{Example}

\begin{Example}\label{ex 4'}
Note that we could use $B$ instead $\Im v$. In this case we choose $f = Y_1^3-Y_2^3$ and take $M = 3Y_2^2$ and $1 \in \big((f):I\big)$. Therefore we obtain $P'=3Y_2^2$, $d'=x_2^2$, $d=x_2^4$ and the next computations are harder as we will see in the Examples \ref{ex 4'.2} and \ref{ex 4'.3}.
\end{Example}

\begin{Remark} We would like to work above with $A''=\fracs{\mathbb C\power{x_1, x_2}}{(x_1^3-x_2^2)}$ instead $A'$, $v$ being given by $v(Y_2)=a_1\rho x_2$. But this is hard since we cannot work in \sing\ with an infinite set of parameters.
We have two choices. If the definition of $v$ involves only a finite set of parameters then we proceed as Example \ref{ex 4} using some $A'\supset \Im v.$
Otherwise, we will see later that in the computation of the General Neron Desingularization we may use only a finite number of the coefficients of the formal power series defining $v(Y)$ and so this computation  works in \sing.
\end{Remark}

\begin{Remark} As we may see our algorithm could go also when $A'$ is not a domain, but there exist $P\in M((f):I)$ as above and a regular element $d\in m$ with $d\equiv P\ \mbox{modulo}\ I$. If $A$ is Cohen-Macaulay we may reduce to the case when there exists a regular element $d\in H_{B/A}\cap A$. However, it is hard usually to reduce to the case when $d\equiv P\ \mbox{modulo}\ I$ for some $P\in M((f):I)$. Sometimes this is possible as shows the following example.
\end{Remark}

\begin{Example}\label{ex 7} Let  $a_1,a_2\in \mathbb C$  be two elements algebraically independent over $\mathbb Q$. Consider $A = \left(\fracs{\mathbb Q[x_1, x_2,x_3]}{(x_2^3-x_3^2,x_1^3-x_3^2)}\right)_{(x_1,x_2,x_3)}$ and $B = \fracs{A[Y_1,Y_2,Y_3]}{(Y_1^3-Y_2^3)}$, $K' =  \fracs{\mathbb Q(a_1,a_2)[a_3]}{(a_3^2-a_1a_2)}$, $A' = \fracs{K'\power{x_1,x_2,x_3}}{(x_2^3-x_3^2,x_1^3-x_3^2)}$ and the map $v$ defined as
$$\begin{xy}
\xymatrix@R-2pc{
v: & B \ar[r]& A' \\
&Y_1 \ar@{|->}[r] & a_3x_1\\
&Y_2 \ar@{|->}[r] & a_3x_2\\
&Y_3 \ar@{|->}[r] & a_1 \Sum{i=0}{30}\fracs{x_3^i}{i!} + a_2\Sum{i=31}{50}\fracs{x_3^i}{i!}
}
\end{xy}$$

Then $\Im\ v$, the new $B$, will be $\fracs{B}{\Ker\ v}$, where the kernel is generated by six polynomials:
\begin{verbatim}
ker[1]=x2*Y1-x1*Y2
ker[2]=Y1^3-Y2^3
ker[3]=x1*Y1^2-x2*Y2^2
ker[4]=x1^2*Y1-x2^2*Y2
ker[5]=x1*x2^2*Y2-x3^2*Y1
ker[6]=x1^2*x2*Y2^2-x3^2*Y1^2
\end{verbatim}
Next we choose $f = x_2Y_1-x_1Y_2$ and we have $M = -x_1$. We may take  $N=-x_3^2 \in \big((f):I\big)$ and  $P' = x_1x_3^2$. Note that $x_1-x_2$ is a zero divisor in $A$ but $d' = P'$ is regular in $A$. In this example we may take $d=d'=P'=P$.
\end{Example}

\begin{Remark}\label{rem 8} Changing $B$ by $\Im v$ can be a hard goal if let us say $A'$ is a factor of the power series ring over $\mathbb C$ in some variables $x$ and $v(Y)$ is defined by formal power series whose coefficients form an infinite field extension $F$ of $\mathbb Q$. If  $v(Y)$ are polynomials in $x$ as in Examples \ref{ex 4}, \ref{ex 7} then it is trivial to find the General Neron Desingularization of $B$ as we explained already
in the last sentences of Section 1. For instance in Example \ref{ex 7}, $B'$ could be a localization of $K' \otimes_{\mathbb Q} A$. Thus  Examples \ref{ex 4}, \ref{ex 7} have no real importance, they being useful only for an illustration of our algorithm. This is the reason that in the next examples   the field $L$ obtained by adjoining to $k$ of all coefficients of $y$ will be an infinite type field extension of $k$ and $v(Y)$ are not all polynomials in $x$.

 However, this will complicate the algorithm because we are not able to tell to  the computer who is $v(Y)$ and so  how to get $d'$. We may choose an element $a\in m$ and find a minimal $c\in \mathbb N$ such that $a^c\in (v(M))+(a^{2c}) $ (this is possible because $\dim A=1$). Set $d'=a^c$. It follows that $d'\in (v(M))+ (d'^2)\subset  (v(M))+ (d'^4)\subset \ldots $ and so $d'\in (v(M))$, that is $d'=v(M)z$ for some $z\in A'$. Certainly we cannot find precisely $z$ but later it is enough  to know just a kind of truncation of it modulo $d'^6$.
 \end{Remark}

\begin{Example} \label{exc} Let  $a_i\in \mathbb C$, $i\in \mathbb N$  be  elements algebraically independent over $\mathbb Q$ and  $\rho$  a root of the polynomial $T^2+T+1$ in $ \mathbb C$.  Let
   $A = \left(\fracs{\mathbb Q[x_1, x_2]}{(x_1^3-x_2^2)}\right)_{(x_1,x_2)}$ and $B = \fracs{A[Y_1,Y_2,Y_3]}{(Y_1^2+Y_1Y_2+Y_2^2)}$,  $A' = \fracs{\mathbb C\power{x_1,x_2}}{(x_1^3-x_2^2)}$ and the map $v$ defined as
$$\begin{xy}
\xymatrix@R-2pc{
v: & B \ar[r]& A' \\
&Y_1 \ar@{|->}[r] & a_1\left(x_2+\Sum{i\geq 7}{} a_ix_2^i\right)\\
&Y_2 \ar@{|->}[r] & a_1a_3\left(x_2+\Sum{i\geq 7}{} a_ix_2^i\right)\\
&Y_3 \ar@{|->}[r] & a_1 \Sum{i=0}{9}\fracs{x_1^i}{i!} + x_2\sum_{i=10}^{\infty}a_{i-8}\fracs{x_1^i}{i!}
}
\end{xy}$$

As in Example \ref{ex 4} we may take $d'=x_2$, $d=d'^2$ and $a$. Our algorithm  goes exactly as in Examples \ref{ex 4}, \ref{ex 14}, \ref{ex 21} providing the same General Neron Desingularization.
This time we cannot find an easy General Neron Desingularization as in the first part of Example \ref{ex 4}.
\end{Example}

 \begin{Example}  \label{ex 10}  Let $A=\fracs{\mathbb Q[x_1,x_2]_{(x_1,x_2)}}{(x_1^2-x_2^3)}$,   $A'=\fracs{\mathbb C\power{x_1,x_2}}{(x_1^2-x_2^3)}$. Then the inclusion $A\subset A'$ is regular. Let $\theta_i = \Sum{j=0}{\infty}\alpha_{ij}x_2^j+x_1 \Sum{j=0}{\infty}\beta_{ij}x_2^j\in \mathbb C\power{x_1,x_2}$ for $i=3,4$ with $\alpha_{i0} =1$  and $y_1=\fracs{\theta_3^3}{\theta_4^2}$, $y_2=\fracs{\theta_4^2}{\theta_3}$, $y_3=x_2\theta_3$, $y_4=x_2\theta_4$. Let $f_1=Y_3^2-x_2^2Y_1Y_2$,  $f_2=Y_4^2-x_2Y_2Y_3$ be polynomials in $A[Y]$, $Y=(Y_1,\ldots,Y_4)$ and set $B=A[Y]/(f)$, $f=(f_1,f_2)$.

 If $R$ is a domain and $u\in R$ is such that $Y^2-u\in R[Y]$ has no solutions in $Q(R)$ then it is easy to see that $R[Y]/(Y^2-u)$ is a domain too. In our case we get that $R=A[Y_1,Y_2,Y_3]/(f_1)$ and $B=R[Y_4]/(f_2)$ are domains too.
 Then the map $v:B\to A'$ given by $Y\to y=(y_1,\ldots,y_4)$ is injective if we suppose that $\theta_3, \theta_4$ are algebraically independent over $A$. This follows since  $B$ is a domain and $\dim B=\trdeg_{Q(A)} Q(B)=\trdeg_{Q(A)}Q(\Im v)=2=\dim \Im v$. Moreover we will assume that the fields $L_i=\mathbb Q\big((\alpha_{ij},\beta_{ij})_j\big)$, $i=3,4$ have infinite transcendental degree over $\mathbb Q$.
 The Jacobian matrix $\left(\fracs{\partial f}{\partial Y}\right)$ have a $2\times 2-$minor $M=\det\left(\fracs{\partial f_i}{\partial Y_j}\right)_{\substack{1\leq i\leq 2\\ 3\leq j\leq 4}}=4Y_3Y_4\not \in (f)$. Note that $v(M)=x_2^2y_5$, where $y_5=1/(4\theta_3 \theta_4)$. Then we may take $B_1=B[Y_5]/(f_3)$, $f_3=-x_2^2+MY_5$ and $v_1$ given by $Y_5\to y_5$. Clearly, $P = M^2Y_5^2 \in H_{B_1/A}$ and $0 \neq d =x_2^4= v_1(P) \in A$.
 \end{Example}

\section{Proof of the case  when $H_{B/A}\cap A\not =0$.}

Thus we may suppose that there exists  $f=(f_1,\ldots,f_r)$, $r\leq n$ a system of polynomials from $I$, a $r\times r$-minor $M$ of the Jacobian matrix $(\partial f_i/\partial Y_j)$ and $N\in ((f):I)$ such that $0\not = d\equiv MN\ \mbox{modulo}\ I$.  Set   $\bar A=A/(d^3)$, $\bar A'=A'/d^3A'$, $\bar u=\bar A\otimes_Au$, $\bar B=B/d^3B$, $\bar v=\bar A\otimes_Av$. Clearly, $\bar u$ is a regular morphism of Artinian  local rings.

\begin{Remark} The whole proof could work with $\bar A=A/d^2u$ for any $u\in m$. We prefer to take $u=d$ as is done in \cite{P} and \cite{P2} but  we could  choose $u\not =d$, $u\in m\setminus m^2$ for easy computations.
\end{Remark}

 By  \cite[19,7.1.5]{G}  for every field extension $L/k$ there exists    a flat complete Noetherian local $\bar A$-algebra $\tilde A$, unique  up to an isomorphism, such that $m\tilde A$ is the maximal ideal of $\tilde A$ and $\tilde A/m\tilde A\cong L$. It follows that $\tilde A$ is Artinian. On the other hand, we may consider the localization $A_L$ of $L\otimes_k\bar A$ in $m(L\otimes_k\bar A)$ which is Artinian and so complete. By uniqueness we see that $A_L\cong \tilde A$. Set $k'=A'/mA'$. It  follows that $\bar A'\cong A_{k'}$. Note that $A_L$ is essentially smooth over $A$ by base change and $\bar A'$ is a filtered union of sub-$\bar A$-algebras $A_L$ with $L/k$ finite type field sub extensions of $k'/k$.

 Let $v$ be given by $Y\to y\in A'^n$.  Choose $L/k$ a finite type field extension such that $A_L$ contains the residue class ${\bar y}\in \bar A'^n$ induced by $y$.
 In fact ${\bar y}$ is a vector of polynomials in the generators of $m$ with the coefficients $c_{\nu}$ in $k'$ and we may take $L=k((c_{\nu})_{\nu})$. Then $\bar v$ factors through $A_L$. Assume that $k[(c_{\nu})_{\nu}]\cong k[(U_{\nu})_{\nu}]/\bar J$ for some new variables $U$ and a prime ideal  $\bar J\subset k[U]$. We have $H_{L/k}\not= 0$ because $L/k$ is separable. Then we may assume that there exist $ \omega=( \omega_1,\ldots,\omega_p)$ in $\bar J$ such that
 $\rho=\det(\partial \omega_i/\partial U_{\nu})_{i,\nu\in [p]}\not =0$ and a nonzero polynomial $ \tau\in ((\omega):\bar J)\setminus \bar J$. Thus $L$ is a fraction ring of the smooth $k$-algebra $(k[U]/( \omega))_{ \rho \tau}$. Note that $\omega, \rho, \tau$ can be considered in $ A$ because $k\subset  A$ and $c_{\nu}\in A'$ because $k'\subset A'$.

   Then $\bar v$ factors through a smooth $\bar A$-algebra $C\cong (\bar A[U]/(\omega))_{\rho\tau\gamma}$ for some polynomial $\gamma$ which is not in $m (\bar A[U]/(\omega))_{\rho\tau}$.

\begin{Lemma}\label{D} There exists a smooth $A$-algebra $D$ such that $\bar v$ factors through $\bar D=\bar A\otimes_A D$.
\end{Lemma}
\begin{proof} By our assumptions $u(k)\subset k'$. Set $D=( A[U]/(\omega))_{\rho\tau\gamma}$ and $w:D\to A'$ be the map given by $U_{\nu}\to c_{\nu}$. We have $C\cong {\bar A}\otimes_AD$.
Certainly, $\bar v$ factors through  $\bar w=\bar A\otimes_Aw $ but in general $v$ does not factor through $w$.
\hfill\ \end{proof}
\begin{Remark} \label{rem} If $A'=\hat A$ then $\bar A\cong \bar A'$ and we may take $D=A$.
\end{Remark}

\begin{Remark} Suppose that $k\subset A$ but $L\not \subset A'$ and so $k'\not \subset A'$.  Then $D=(A[U,Z]/(\omega-d^3Z))_{\rho\tau\gamma}$, $Z=(Z_{\nu})$ is a smooth $A$-algebra and $\bar D\cong C[Z]$. Since  $\bar v$ factors through a map $C\to \bar A'$ given let us say by $U\to \lambda+d^3A'$ for some $\lambda$ in $A'$ we see that  $\omega(\lambda)\equiv 0$ modulo $d^3$, that is $\omega(\lambda)=d^3z$ for some $z$ in $A'$.
Let $w:D\to A'$ be the $A$-morphism given by $(U,Z)\to (\lambda,z)$.  Certainly, $\bar v$ factors through  $\bar w=\bar A\otimes_Aw $ but in general $v$ does not factor through $w$.
If also $k\not \subset A$ then the construction of $D$ goes as above but using a lifting of $\omega,\tau, \gamma$ from $k[U]$ to $A[U]$. In both cases we may use $D$ as it follows.
\end{Remark}

\begin{Example}\label{ex 14} We reconsider Example \ref{ex 4}. We already know that $d = x_2^2$.
The algorithm gives us the following output:
\begin{verbatim}
This is C:
//   characteristic : 0
//   number of vars : 5
//        block   1 : ordering dp
//                  : names    a1 a3 a x1 x2
//        block   2 : ordering C
// quotient ring from ideal
_[1]=3*a1*a+2*a3+1
_[2]=a3^2+a3+1
_[3]=x1^3-x2^2
_[4]=x2^6
This is D:
//   characteristic : 0
//   number of vars : 5
//        block   1 : ordering dp
//                  : names    a1 a3 a x1 x2
//        block   2 : ordering C
// quotient ring from ideal
_[1]=3*a1*a+2*a3+1
_[2]=a3^2+a3+1
_[3]=x1^3-x2^2
\end{verbatim}
Indeed, $$C = \fracs{\bar A[a_1,a_3, a]}{\left(3a_1a+2a_3+1,a_3^2+a_3+1,x_2^6\right)}$$ and $$D = \fracs{ A[a_1,a_3, a]}{\left(3a_1a+2a_3+1,a_3^2+a_3+1\right)}.$$

Note that the first polynomial from $C$ comes from the standard basis computation of the ideal $\left(2a_1a_3a+a_1a-1,a_3^2+a_3+1\right)$.
\end{Example}

\begin{Example}\label{ex 4'.2}
Now we reconsider Example \ref{ex 4'}. We know that $d = x_2^4$.
The algorithm gives us the following output:
\begin{verbatim}
This is C:
//   characteristic : 0
//   number of vars : 5
//        block   1 : ordering dp
//                  : names    a1 a3 a x1 x2
//        block   2 : ordering C
// quotient ring from ideal
_[1]=a3^2+a3+1
_[2]=x1^3-x2^2
_[3]=a1^2*a-a3
_[4]=x2^12

This is D:
//   characteristic : 0
//   number of vars : 5
//        block   1 : ordering dp
//                  : names    a1 a3 a x1 x2
//        block   2 : ordering C
// quotient ring from ideal
_[1]=a3^2+a3+1
_[2]=x1^3-x2^2
_[3]=a1^2*a-a3
\end{verbatim}
Indeed, $$C = \fracs{\bar A[a_1,a_3, a]}{\left(a_3^2+a_3+1,a_1^2a-a_3,x_2^{12}\right)}$$ and $$D = \fracs{ A[a_1,a_3, a]}{\left(a_3^2+a_3+1,a_1^2a-a_3\right)}.$$
\end{Example}

\begin{Example}\label{ex 15} In the case of Example \ref{ex 7} we obtain the following output:
\begin{verbatim}
This is C:
//   characteristic : 0
//   number of vars : 5
//        block   1 : ordering dp
//                  : names    a1 a3 x1 x2 x3
//        block   2 : ordering C
// quotient ring from ideal
_[1]=x2^3-x3^2
_[2]=x1^3-x3^2
_[3]=x3^8
This is D:
//   characteristic : 0
//   number of vars : 5
//        block   1 : ordering dp
//                  : names    a1 a3 x1 x2 x3
//        block   2 : ordering C
// quotient ring from ideal
_[1]=x2^3-x3^2
_[2]=x1^3-x3^2
\end{verbatim}

Indeed this is the case since we have $d = x_1x_3^2$ and hence $$C = \fracs{\bar A[a_1,a_3]}{\left(x_3^8\right)}$$ and $$D = A[a_1, a_3].$$
\end{Example}

  \begin{Example}\label{ex 16} In Example \ref{ex 10} we consider  $a_1,a_2$ algebraically independent over $\mathbb Q$ and  set  $\theta_3' = 1+a_1x_2$ and
  $\theta_4' = 1+a_2x_2^2$. Suppose that $\theta_i'\equiv \theta_i\ \mbox{modulo}\ x_2^{12}$.
We have $y_3=x_2\theta_3$, $y_4=x_2\theta_4$, $y_1=\fracs{\theta_3^3}{\theta_4^2}$, $y_2=\fracs{\theta_4^2}{\theta_3}$, $y_5=\fracs{1}{(4\theta_3 \theta_4)}$. Choose $y'_i$, $i\in [5]$ polynomials with degrees $\leq 11$ in $x_2$ and linear in $x_1$ such that $y'_i\equiv y_i\mod (x_1^2,x_2^{12})$. We get $y'_1\equiv y_1=\theta_3^3/\theta_4^2\equiv {\theta'}_3^3/{\theta'}_4^2 \mod  (x_1^2,x_2^{12})$ and similarly for $y'_i$, $i>1$. Here we use the fact that $\theta_4^{-2}=\Sum{j=1}{e} (1-\theta_4^2)^j$ for some $e>>0$ because $1-\theta_4^2$ is nilpotent in the ring $\bar A[a_1, a_2, a_3, a_4]$. Thus the coefficients of $y'_i$, $i\in [5]$ belong to the field $L$ obtained by adjoining to   $\mathbb Q$ the coefficients of $\theta'_3,\theta'_4$. Note that in this case
$L= Q\left(\mathbb Q[a_1,\ldots,a_4]\right)$.
Then we obtain the following output:
\begin{verbatim}
This is C:
//   characteristic : 0
//   number of vars : 4
//        block   1 : ordering dp
//                  : names    a1 a2 x1 x2
//        block   2 : ordering C
// quotient ring from ideal
_[1]=x2^3-x1^2
_[2]=x1^8
This is D:
//   characteristic : 0
//   number of vars : 4
//        block   1 : ordering dp
//                  : names    a1 a2 x1 x2
//        block   2 : ordering C
// quotient ring from ideal
_[1]=x2^3-x1^2
\end{verbatim}

 Thus $ C=\fracs{\bar A[a_1,\ldots, a_4]}{(x_2^{12})}\cong {\bar A}[a_1,\ldots,a_4]$ which is smooth over $\bar A$. Then $D$ is equal with $A[a_1,\ldots,a_4]$.
\end{Example}

Back to our proof note that the composite map $\bar B\to C\to \bar D$ is given by $Y\to y'+d^3D$ for some $y'\in D^n$. Thus $I(y')\equiv 0$ modulo $d^3D$. Since $\bar v$ factors through $\bar w$ we see that $\bar w(y'+d^3D)=\bar y$. Set $\tilde y=w(y')$. We get
  $y-\tilde y\in d^3A'^n$, let us say $y-\tilde y=d^2\epsilon$ for $\epsilon\in dA'^n$.

We have $d\equiv P$ modulo $I$ and so $P(y')\equiv d$ modulo $d^3$ in $D$ because $I(y')\equiv 0$ modulo $d^3D$. Thus $P(y')=ds$ for a certain $s\in D$ with $s\equiv 1$ modulo $d$.
Assume that $P=NM$ for some $N\in ((f):I)$.  Recall from beginning of Section 2  that the new $M$ is now the old one multiplied with $P'$ and the new  $N$ is the old one multiplied with $Z^2$. Let  $H$ be the $n\times n$-matrix obtained adding down to $(\partial f/\partial Y)$ as a border the block $(0|\mbox{Id}_{n-r})$. Let $G'$ be the adjoint matrix of $H$ and $G=NG'$. We have
$$GH=HG=NM \mbox{Id}_n=P\mbox{Id}_n$$
and so
$$ds\mbox{Id}_n=P(y')\mbox{Id}_n=G(y')H(y').$$

 Then $t:=H(y')\epsilon\in d{A'}^n$
satisfies
$$G(y')t=P(y')\epsilon=ds\epsilon$$
 and so
 $$s(y-\tilde y)=dw(G(y'))t.$$
 Let
 \begin{equation}\label{def of h}h=s(Y-y')-dG(y')T,\end{equation}
 where  $T=(T_1,\ldots,T_n)$ are new variables. The kernel of the map
$\phi:D[Y,T]\to A'$ given by $Y\to y$, $T\to t$ contains $h$. Since
$$s(Y-y')\equiv dG(y')T\ \mbox{modulo}\ h$$
and
$$f(Y)-f(y')\equiv \sum_j\fracs{\partial f}{\partial Y_j}(y') (Y_j-y'_j)$$
modulo higher order terms in $Y_j-y'_j$ by Taylor's formula we see that for $p=\max_i \deg f_i$ we have
\begin{equation}\label{def of Q}s^pf(Y)-s^pf(y')\equiv  \sum_js^{p-1}d\fracs{\partial f}{\partial Y_j}(y') G_j(y')T_j+d^2Q=s^{p-1}dP(y')T+d^2Q\end{equation}
modulo $h$ where $Q\in T^2 D[T]^r$. This is because $(\partial f/\partial Y)G=(P\mbox{Id}_r|0)$.  We have $f(y')=d^2b$ for some $b\in dD^r$. Then
\begin{equation}\label{def of g}g_i=s^pb_i+s^pT_i+Q_i, \qquad i\in [r] \end{equation}  is in the kernel of $\phi$ because $d^2\phi(g)=d^2g(t)\in (h(y,t),f(y))=(0)$. Set $E=D[Y,T]/(I,g,h)$ and let  $\psi:E\to A'$ be the map induced by $\phi$. Clearly, $v$ factors through $\psi$ because $v$ is the composed map $B\to B\otimes_AD\cong D[Y]/I\to E\xrightarrow{\psi} A'$.

Note that the $r\times r$-minor  $s'$ of $(\partial g/\partial T)$ given by the first  $r$-variables $T$ is from $s^{rp}+(T)\subset 1+(d,T)D[Y,T]$ because $Q\in (T)^2$. Then $U=(D[Y,T]/(h,g))_{ss'}$ is smooth over $D$. We claim that $I\subset (h,g)D[Y,T]_{ss's''}$ for some other $s''\in 1+(d,T)D[Y,T]$. Indeed, we have $PI\subset (h,g)D[Y,T]_s$ and so $P(y'+s^{-1}dG(y')T)I\subset (h,g)D[Y,T]_s$. Since  $P(y'+s^{-1}dG(y')T)\in P(y')+d(T)$ we get $P(y'+s^{-1}dG(y')T)=ds''$ for some $s''\in 1+(d,T)D[Y,T]$. It follows that $s''I\subset (h,g)D[Y,T]_{ss'}$ because $d$ is regular in $U$, the map $D\to U$ being flat, and so $I\subset (h,g)D[Y,T]_{ss's''}$. Thus $E_{ss's''}\cong U_{s''} $ is a $B$-algebra which is also standard smooth over $D$ and $A$.

 As $w(s)\equiv 1$ modulo $d$ and $w(s'),w(s'')\equiv 1$ modulo $(d,t)$, $d,t\in mA'$ we see that $w(s),w(s'), w(s'')$ are invertible because  $A'$ is local  and $\psi$ (thus $v$) factors through the standard smooth $A$-algebra $E_{ss's''}$.

\section{A theorem of  Greenberg's type}

Let $(A,m)$ be a Cohen-Macaulay local ring (for example a reduced ring) of dimension one, $A'={\hat A}$ the completion of $A$, $B=A[Y]/I$, $Y=(Y_1,\ldots,Y_n)$ a finite type $A$-algebra and $c,e\in \mathbb N$.
Suppose that there exist  $f=(f_1,\ldots,f_r)$ in $I$, a $r\times r$-minor $M$ of the Jacobian matrix $(\partial f/\partial Y)$, $N\in ((f):I) $ and an $A$-morphism $v:B\to A/m^{2e+c}$
such that $(v(MN))\supset m^e/m^{2e+c}$.

\begin{Theorem} \label{gr} Then there exists an $A$-morphism $v':B\to {\hat A}$ such that $v'\equiv v\ \mbox{modulo}\ m^c$, that is $v'(Y+I)\equiv v(Y+I)\ \mbox{modulo}\ m^c$.
\end{Theorem}
\begin{proof} We note that the proof of Theorem \ref{m} can work somehow in this case. Let $y'\in A^n$ be an element inducing $v(Y+I)$. Then $m^e\subset ((MN)(y'))+m^{2e+c}\subset ((MN)(y'))+m^{3e+2c}\subset \ldots$ by hypothesis. It follows that  $m^e\subset ((MN)(y'))$.
Since $A$ is Cohen-Macaulay we see that $m^e $ contains a regular element of $A$ and so $(MN)(y')$ must be regular too.

Set $d=(MN)(y')$.
Next we follow the proof of Theorem \ref{m} with $D=A$, $s=1$, $P=MN$ and $H$, $G$ such that .
$$d\mbox{Id}_n=P(y')\mbox{Id}_n=G(y')H(y').$$
Let  $$h=Y-y'-dG(y')T,$$
 where  $T=(T_1,\ldots,T_n)$ are new variables. We have
  $$f(Y)-f(y')\equiv
dP(y')T+d^2Q$$
modulo $h$ where $Q\in T^2 A[T]^r$. But $f(y')\in m^{2e+c}A^r\subset d^2m^cA^r$ and we get $f(y')=d^2b$ for some $b\in m^cA^r$. Set $g_i=b_i+T_i+Q_i$, $i\in [r]$ and $E=A[Y,T]/(I,h,g)$. We have an $A$-morphism $\beta:E\to A/m^c$ given by $(Y,T)\to (y',0)$  because $I(y')\equiv 0\ \mbox{modulo}\ m^{2e+c}$, $h(y',0)=0$ and $g(0)=b\in m^cA^r$.

 As in the proof of Theorem \ref{m} we have $E_{s's''}\cong U_{s''}$, where $U=(A[Y,T]/(g,h))_{s'}$. This isomorphism follows because $d$ is regular in $A$ and so in $U$. Consequently, $E_{s's''}$ is smooth over $A$.
 Note that $\beta$ extends to a map $\beta':E_{s's''}\to A/m^c$. By the Implicit Function Theorem    $\beta'$ can be lifted to a map $w:E_{s's''}\to \hat A$ which coincides with $\beta'$ modulo $m^c$. It follows that the composite map $v'$, $B\to  E_{s's''}\xrightarrow{w} \hat A$ works.
\hfill\ \end{proof}

\begin{Corollary} In the assumptions of the above theorem, suppose that $(A,m)$ is excellent Henselian. Then there exists an $A$-morphism $v'':B\to A$ such that $v''\equiv v\ \mbox{modulo}\ m^c$, that is $v''(Y+I)\equiv v(Y+I)\ \mbox{modulo}\ m^c$.
\end{Corollary}
\begin{proof}
An excellent Henselian local ring $(A,m)$ has the property of Artin approximation by \cite{P}, that is the solutions in $A$ of a system of polynomial equations $f$ over $A$ are dense in the set of the solutions of $f$ in $\hat A$. By Theorem \ref{gr} we get an $A$-morphism $v':B\to \hat A$ such that $v'\equiv v \ \mbox{modulo} \ m^c$. Then there exists an $A$-morphism $v'':B\to A$ such that $v''\equiv v'\equiv v\ \mbox{modulo}\ m^c$ by the property of Artin approximation.
\hfill\ \end{proof}

\begin{Theorem} \label{gr1} Let $(A,m)$ be a Cohen-Macaulay local ring  of dimension one,  $B=A[Y]/I$, $Y=(Y_1,\ldots,Y_n)$ a finite type $A$-algebra, $e\in \mathbb N$ and  $f=(f_1,\ldots,f_r)$ a system of polynomials from $I$. Suppose  that $A$ is excellent Henselian and there exist
a $r\times r$-minor $M$ of the Jacobian matrix $(\partial f/\partial Y)$, $N\in ((f):I) $ and $y'\in A^n$ such that $I(y')\equiv 0\ \mbox{modulo}\ m^e$ and  $((NM)(y'))\supset m^e$. Then the following statements are equivalent:
\begin{enumerate}
\item{} there exists $y''\in A^n$ such that $I(y'')\equiv 0\ \mbox{modulo}\ m^{3e}$ and $y''\equiv y'\ \mbox{modulo}\ m^e$,
\item{}  there exists $y\in A^n$ such that $I(y)=0$ and $y\equiv y'\ \mbox{modulo}\ m^e$.
\end{enumerate}
\end{Theorem}
For the proof apply the above corollary and Theorem \ref{gr}.

\section{Computation of the General Neron Desingularization in Examples \ref{ex 4}, \ref{ex 4'}, \ref{ex 7}, \ref{ex 10} }

\begin{Example}\label{tutorial}
We would like to compute Example \ref{ex 4} in \sing \ using \verb"GND.lib". We quickly recall the example.

Let  $a_1,a_2\in \mathbb C$  be two elements algebraically independent over $\mathbb Q$ and  $\rho$  a root of the polynomial $T^2+T+1$ in $ \mathbb C$. Then $k'=\fracs{\mathbb Q(a_1, a_2)[a_3]}{(a_3^2+a_3+1)}\cong \mathbb Q(\rho,a_1,a_2) $. Let
   $A = \left(\fracs{\mathbb Q[x_1, x_2]}{(x_1^3-x_2^2)}\right)_{(x_1,x_2)}$ and $B = \fracs{A[Y_1,Y_2,Y_3]}{(Y_1^3-Y_2^3)}$,  $A' = \fracs{k'\power{x_1,x_2}}{(x_1^3-x_2^2)}$ and the map $v$ defined as
$$\begin{xy}
\xymatrix@R-2pc{
v: & B \ar[r]& A' \\
&Y_1 \ar@{|->}[r] & a_1x_2\\
&Y_2 \ar@{|->}[r] & a_1a_3x_2\\
&Y_3 \ar@{|->}[r] & a_1 \Sum{i=0}{30}\fracs{x_1^i}{i!} + a_2x_2\Sum{i=31}{50}\fracs{x_1^i}{i!}
}
\end{xy}$$

 For this we do the following:
\begin{verbatim}
LIB "GND.lib";                              //load the library
ring All = 0,(a1,a2,a3,x1,x2,Y1,Y2,Y3),dp;  //define the ring
int nra = 3;                                //number of a's
int nrx = 2;                                //number of x's
int nry = 3;                                //number of Y's
ideal xid = x1^3-x2^2;                      //define the ideal from A
ideal yid = Y1^3-Y2^3;                      //define the ideal from B
ideal aid = a3^2+a3+1;                      //define the ideal from k'
poly y;
int i;
for(i=0;i<=30;i++)
{
  y = y + a1*x1^i/factorial(i);
}
for(i=31;i<=50;i++)
{
  y = y + a2*x2*x1^i/factorial(i);
}
ideal f = a1*x2,a1*a3*x2,y;                //define the map v
desingularization(All, nra,nrx,nry,xid,yid,aid,f,"debug");
\end{verbatim}
\end{Example}

 \begin{Example} \label{ex 21} We continue on the idea of Examples \ref{ex 4}, \ref{ex 14}. The bordered matrix $H$ defined above is equal to $$ H = \left(
 \begin{array}{cccc}
 2Y_1+Y_2 & Y_1+2Y_2 & 0 & 0\\
0 &      0 &      1&0 \\
1 &      0 &      0 & 0 \\
Y_4 &      2Y_4 &    0 & Y_1+2Y_2\\
 \end{array}\right)$$ and hence $G = N\cdot G'$ is equal to $$ G = Y_4^2 \cdot \left(
 \begin{array}{cccc}
 0 & 0 & Y_1^2+4Y_1Y_2+4Y_2^2 & 0\\
Y_1+2Y_2 &      0 &      -2Y_1^2-5Y_1Y_2-2Y_2^2&0 \\
0 &      Y_1^2+4Y_1Y_2+4Y_2^2 &      0 & 0 \\
-2Y_4 &     0 &   3Y_1Y_4 & Y_1+2Y_2\\
 \end{array}\right)$$ and $s = 1$. Using the definition of $h$ in Equation \ref{def of h}, we get that $$\begin{array}{ccl}
 h_1&=&Y_1-\left(x_2^4\right) \cdot T_3-(a_1x_2),\\

h_2&=&Y_2-\fracs{x_2^3}{2a_1a_3+a_1}\cdot T_1+  \fracs{a_3x_2^4+2x_2^4}{2a_3+1}\cdot T_3-(a_1a_3x_2),\\

h_3&=&Y_3-\left(x_2^4\right) \cdot T_2-\left(
\fracs{1}{6!}a_1x_1^6+\fracs{1}{5!}a_1x_1^5+\fracs{1}{4!}a_1x_1^4+\fracs{1}{3!}a_1x_1^3+\fracs{1}{2}a_1x_1^2+\right.\\&&\left.a_1x_1+a_1\right),\\

h_4&=&Y_4+\fracs{2x_2^2}{\left(2a_1a_3+a_1\right)^3}\cdot T_1-\fracs{3x_2^3}{a_1^2\left(2a_3+1\right)^3}\cdot T_3-\fracs{x_2^3}{2a_1a_3+a_1} \cdot T_4 +\fracs{1}{2a_1a_3+a_1}.\end{array}$$

From Equation \ref{def of Q} we get that
$$\begin{array}{ccl}
Q_1&=&\fracs{x_2^2}{\left(2a_1a_3+a_1\right)^2}\cdot T_1^2-
\fracs{3x_2^3}{a_1\left(2a_3+1\right)^2} \cdot T_1T_3+
\fracs{3a_3^2x_2^4+3a_3x_2^4+3x_2^4}{\left(2a_3+1\right)^2} \cdot T_3^2, \\
Q_2&=&-\fracs{4x_2}{\left(2a_1a_3+a_1\right)^4}\cdot T_1^2+
\fracs{12x_2^2}{a_1^3 \left(2a_3+1\right)^4} \cdot T_1T_3 -  \fracs{9x_2^3}{a_1^2\left(2a_3+1\right)^4}\cdot T_3^2+\\&& \fracs{2x_2^2}{\left(2a_1a_3+a_1\right)^2}\cdot T_1T_4- \fracs{3x_2^3}{a_1\left(2a_3+1\right)^2}\cdot T_3T_4
\end{array}$$

and therefore following the definition of $g$ in Equation \ref{def of g} we have
$$\begin{array}{ccl}
g_1&=&Q_1 +T_1+(a_1^2a_3^2+a_1^2a_3+a_1^2),\\
g_2&=& Q_2+T_2.
\end{array}$$

We print now the algorithm's debug output using the line codes from Example \ref{tutorial}.
\begin{verbatim}
This is the bordered matrix H:
2*Y1+Y2,Y1+2*Y2,0,0,
0,      0,      1,0,
1,      0,      0,0,
Z,      2*Z,    0,Y1+2*Y2
This is G:
0,                0,     G[1,3],   0,
Y1*Y4^2+2*Y2*Y4^2,0,     G[2,3],   0,
0,                G[3,2],0,        0,
-2*Y4^3,          0,     3*Y1*Y4^3,Y1*Y4^2+2*Y2*Y4^2

G[1,3]=Y1^2*Y4^2+4*Y1*Y2*Y4^2+4*Y2^2*Y4^2
G[2,3]=-2*Y1^2*Y4^2-5*Y1*Y2*Y4^2-2*Y2^2*Y4^2
G[3,2]=Y1^2*Y4^2+4*Y1*Y2*Y4^2+4*Y2^2*Y4^2

s = 1
h =
_[1]=Y1+(-x2^4)*T3+(-a1*x2)
_[2]=Y2+(-x2^3)/(2*a1*a3+a1)*T1+(a3*x2^4+2*x2^4)/(2*a3+1)*T3+
	(-a1*a3*x2)
_[3]=Y3+(-x2^4)*T2+(-a1*x1^6-6*a1*x1^5-30*a1*x1^4-120*a1*x1^3
	-360*a1*x1^2-720*a1*x1-720*a1)/720
_[4]=Y4+(2*x2^2)/(8*a1^3*a3^3+12*a1^3*a3^2+6*a1^3*a3+a1^3)*T1+
	(-3*x2^3)/(8*a1^2*a3^3+12*a1^2*a3^2+6*a1^2*a3+a1^2)*T3+
	(-x2^3)/(2*a1*a3+a1)*T4-1/(2*a1*a3+a1)

m = 2
QT =
_[1]=(x2^2)/(4*a1^2*a3^2+4*a1^2*a3+a1^2)*T1^2+
	(-3*x2^3)/(4*a1*a3^2+4*a1*a3+a1)*T1*T3+
	(3*a3^2*x2^4+3*a3*x2^4+3*x2^4)/(4*a3^2+4*a3+1)*T3^2

_[2]=(-4*x2)/(16*a1^4*a3^4+32*a1^4*a3^3+24*a1^4*a3^2+8*a1^4*a3+a1^4)
	*T1^2+(12*x2^2)/(16*a1^3*a3^4+32*a1^3*a3^3+24*a1^3*a3^2+8*a1^3*a3
	+a1^3)*T1*T3+(-9*x2^3)/(16*a1^2*a3^4+32*a1^2*a3^3+24*a1^2*a3^2
	+8*a1^2*a3+a1^2)*T3^2+(2*x2^2)/(4*a1^2*a3^2+4*a1^2*a3+a1^2)*T1*T4
	+(-3*x2^3)/(4*a1*a3^2+4*a1*a3+a1)*T3*T4

f =
f[1]=Y1^2+Y1*Y2+Y2^2
f[2]=Y1*Y4+2*Y2*Y4+(-x2^2)

g =
_[1]=(x2^2)/(4*a1^2*a3^2+4*a1^2*a3+a1^2)*T1^2+(-3*x2^3)/(4*a1*a3^2+
	4*a1*a3+a1)*T1*T3+(3*a3^2*x2^4+3*a3*x2^4+3*x2^4)/(4*a3^2+4*a3+1)*T3^2
	+T1+(a1^2*a3^2+a1^2*a3+a1^2)
_[2]=(-4*x2)/(16*a1^4*a3^4+32*a1^4*a3^3+24*a1^4*a3^2+8*a1^4*a3+a1^4)
	*T1^2+(12*x2^2)/(16*a1^3*a3^4+32*a1^3*a3^3+24*a1^3*a3^2+8*a1^3*
	a3+a1^3)*T1*T3+(-9*x2^3)/(16*a1^2*a3^4+32*a1^2*a3^3+24*a1^2*a3^2
	+8*a1^2*a3+a1^2)*T3^2+(2*x2^2)/(4*a1^2*a3^2+4*a1^2*a3+a1^2)*T1*T4
	+(-3*x2^3)/(4*a1*a3^2+4*a1*a3+a1)*T3*T4+T2


 \end{verbatim}
 Thus the General Neron Desingularization is a localization of $D[Y,T]/(h,g)\cong D[T]/(g)$.
 \end{Example}

 \begin{Example}\label{ex 4'.3}
 In the case of Example \ref{ex 4'} and \ref{ex 4'.2} we obtain that the bordered matrix $$ H = \left(
 \begin{array}{cccc}
 3Y_1^2 & -3Y_2^2 & 0 & 0\\
0 &      0 &      1&0 \\
1 &      0 &      0 & 0 \\
0 &      -6Y_2Y_4 &    0 & -3Y_2^2\\
 \end{array}\right)$$ and hence $G = N\cdot G'$ is equal to
 $$ G = Y_4^2 \cdot \left(
 \begin{array}{cccc}
 0 & 0 & 9Y_2^4 & 0\\
-3Y_2^2 &      0 &      -3Y_1^2Y_2^2&0 \\
0 &     9Y_2^4 &      0 & 0 \\
6Y_2Y_4 &     0 &   -18Y_1^2Y_2Y_4 & -3Y_2^2\\
 \end{array}\right)$$ and $s = 1$. Using the definition of $h$ in Equation \ref{def of h}, we get that $$\begin{array}{ccl}
 h_1&=&Y_1-\left(x_2^8\right)\cdot T_3+\left(-a_1x_2\right),\\

h_2&=&Y_2+\fracs{x_2^6}{3a_1^2a_3^2}\cdot T_1-\fracs{x_2^8}{a_3^2}\cdot T_3-\left(a_1a_3x_2\right),\\

h_3&=&Y_3-\left(x_2^8\right) \cdot T_2-\left(\fracs{1}{12!}a_1x_1^{12}+\fracs{1}{11!}a_1x_1^{11}+\fracs{1}{10!}a_1x_1^{10}+\fracs{1}{9!}a_1x_1^{9}+\fracs{1}{8!}a_1x_1^{8}\right.\\&&\left.+\fracs{1}{7!}a_1x_1^{7}
+\fracs{1}{6!}a_1x_1^6+\fracs{1}{5!}a_1x_1^5+\fracs{1}{4!}a_1x_1^4+\fracs{1}{3!}a_1x_1^3+\fracs{1}{2}a_1x_1^2+a_1x_1+a_1\right),\\

h_4&=&Y_4+\fracs{2x_2^5}{9a_1^5a_3^5} \cdot T_1 - \fracs{2x_2^7}{3a_1^3a_3^5}\cdot T_3+\fracs{x_2^6}{3a_1^2a_3^2} \cdot T_4 +\fracs{1}{3a_1^2a_3^2}.\end{array}$$

From Equation \ref{def of Q} we get that
$$\begin{array}{ccl}
Q_1&=&\fracs{x_2^{10}}{27a_1^6a_3^6}\cdot T_1^3-
\fracs{x_2^{12}}{3a_1^4a_3^6} \cdot T_1^2T_3+
\fracs{x_2^{14}}{a_1^2a_3^6} \cdot T_1T_3^2+
\fracs{a_3^6x_2^{16}-x_2^{16}}{a_3^6} \cdot T_3^3-\\&&
\fracs{x_2^5}{3a_1^3a_3^3}\cdot T_1^2+
\fracs{2x_2^7}{a_1a_3^3}\cdot T_1T_3+
\fracs{3a_1a_3^3x_2^9-3a_1x_2^9}{a_3^3} \cdot T_3^2, \\

Q_2&=&\fracs{2ax_2^9}{27a_1^9a_3^9} \cdot T_1^3-
\fracs{2x_2^{11}}{3a_175a_3^9} \cdot T_1^2T_3+
\fracs{2x_2^{13}}{a_1^5a_3^9} \cdot T_1T_3^2-
\fracs{2x_2^{15}}{a_1^3a_3^9} \cdot T_3^3+\\&&
\fracs{x_2^{10}}{9a_1^6a_3^6} \cdot T_1^2T_4-
\fracs{2x_2^{12}}{3a_1^4a_3^6} \cdot T_1T_3T_4+
\fracs{x_2^{14}}{a_1^2a_3^6} \cdot T_3^2T_4-
\fracs{x_2^4}{3a_1^6a_3^6} \cdot T_1^2+\\&&
\fracs{2x_2^6}{a_1^4a_3^6} \cdot T_1T_3-
\fracs{3x_2^8}{a_1^2a_3^6} \cdot T_3^2-
\fracs{2x_2^5}{3a_1^3a_3^3} \cdot T_1T_4+
\fracs{2x_2^7}{a_1a_3^3} \cdot T_3T_4
\end{array}$$

and therefore following the definition of $g$ in Equation \ref{def of g} we have
$$\begin{array}{ccr}
g_1&=&Q_1 +  T_1,\\
g_2&=& Q_2 + T_2
\end{array}$$

To obtain this with \sing,  we use the same code lines as in Example \ref{tutorial}, but we change the last one with
\begin{verbatim}
desingularization(All, nra,nrx,nry,xid,yid,aid,f,"injective","debug");
\end{verbatim}
Doing this, the algorithm will not compute the kernel because of the \verb"injective" argument.
 \end{Example}

\begin{Example} We do now the same computations for Examples \ref{ex 7}, \ref{ex 15}.
 The bordered matrix $H$ defined above is equal to $$ H = \left(
 \begin{array}{cccc}
 x_2 & -x_1 & 0 & 0\\
0 &      0 &      1&0 \\
1 &      0 &      0 & 0 \\
0 &      0 &    0 & -x_1x_3^2\\
 \end{array}\right)$$ and hence $G = N\cdot G'$ is equal to $$ G = Y_4^2 \cdot \left(
 \begin{array}{cccc}
 0 & 0 & x_1^2x_3^4 & 0\\
-x_1x_3^4 &      0 &      x_1x_2x_3^4&0 \\
0 &      x_1^2x_3^4 &      0 & 0 \\
0 &     0 &   0 & -x_1x_3^2\\
 \end{array}\right)$$
 and $s = 1$. Using the definition of $h$ in Equation \ref{def of h}, we get that $$\begin{array}{ccl}
 h_1&=&Y_1+\left(x_1^3x_3^6\right) \cdot T_3-(a_3x_1),\\

h_2&=&Y_2-\left(x_1^2x_3^6\right)\cdot T_1+  \left(x_1^2x_2x_3^6\right)\cdot T_3-(a_3x_2),\\

h_3&=&Y_3+\left(x_1^3x_3^6\right) \cdot T_2-\left(\fracs{1}{7!}a_1x_3^7+\fracs{1}{6!}a_1x_3^6+\fracs{1}{5!}a_1x_3^5\right. \\ && \left. +\fracs{1}{4!}a_1x_3^4+\fracs{1}{3!}a_1x_3^3+\fracs{1}{2!}a_1x_3^2+a_1x_3+a_1\right),\\

h_4&=&Y_4+\left(x_1^2x_3^4\right) \cdot T_4 + 1.\end{array}$$

From Equation \ref{def of Q} we get that
$$\begin{array}{ccl}
Q_1&=&0, \\
Q_2&=&0
\end{array}$$
and therefore following the definition of $g$ in Equation \ref{def of g} we have
$$\begin{array}{ccr}
g_1&=&T_1\\
g_2&=&T_2.
\end{array}$$
To compute this with the library we do the following:
\begin{verbatim}
ring All = 0,(a1,a2,a3,x1,x2,x3,Y1,Y2,Y3),dp;
int nra = 3;
int nrx = 3;
int nry = 3;
ideal xid = x2^3-x3^2,x1^3-x3^2;
ideal yid = Y1^3-Y2^3;
ideal aid = a3^2-a1*a2;
poly y;
int i;
for(i=0;i<=30;i++)
{
  y = y + a1*x3^i/factorial(i);
}
for(i=31;i<=50;i++)
{
  y = y + a2*x3^i/factorial(i);
}
ideal f = a3*x1,a3*x2,y;
desingularization(All, nra,nrx,nry,xid,yid,aid,f,"debug");
\end{verbatim}
 The algorithm's output is as expected:
 \begin{verbatim}
This is the nice bordered matrix H:
(x2),(-x1),0,0,
0,   0,    1,0,
1,   0,    0,0,
0,   0,    0,(-x1*x3^2)
This is G:
0,              0,               (x1^2*x3^4)*Y4^2, 0,
(-x1*x3^4)*Y4^2,0,               (x1*x2*x3^4)*Y4^2,0,
0,              (x1^2*x3^4)*Y4^2,0,                0,
0,              0,               0,                (-x1*x3^2)*Y4^2


s = 1

h =
h[1]=Y1+(x1^3*x3^6)*T3+(-a3*x1)
h[2]=Y2+(-x1^2*x3^6)*T1+(x1^2*x2*x3^6)*T3+(-a3*x2)
h[3]=Y3+(x1^3*x3^6)*T2+(-a1*x3^7-7*a1*x3^6-42*a1*x3^5-210*a1*x3^4
     -840*a1*x3^3-2520*a1*x3^2-5040*a1*x3-5040*a1)/5040
h[4]=Y4+(-x1^2*x3^4)*T4+1

m = 1

QT =
QT[1]=0
QT[2]=0
f =
f[1]=(x2)*Y1+(-x1)*Y2
f[2]=(x1*x3^2)*Y4+(-x1*x3^2)

g
_[1]=T1
_[2]=T2
\end{verbatim}

Thus the General Neron Desingularization is a localization of $D[Y,T_3,T_4]/(h)\cong D[T_3,T_4]$.
 \end{Example}

 \begin{Example} We do now the same computations for Example \ref{ex 16}. In this example, the computations are much more complicated. The output is unfortunately too big but we will try to describe the result.

 The bordered matrix $H$ defined above is equal to $$ H = \left(
 \begin{array}{ccccc}
0 & x_2\cdot Y_3 & x_2\cdot Y_2 & -2\cdot Y_4 & 0\\
x_1^2\cdot Y_2 &  x_1^2 \cdot Y_1     &     -2\cdot Y_3 &0&0 \\
0 &      1 &      0 & 0&0 \\
1 &      0 &    0 & 0&0\\
0 &      0 &    4Y_4Y_5 & 4Y_3Y_5& 4Y_3Y_4\\
 \end{array}\right)$$ and hence $G = N\cdot G'$ is equal to $$ G = Y_5^2 \cdot \left(
 \begin{array}{ccccc}
 0 & 0 & 0 & 16Y_3^2Y_4^2&0\\
0 &      0 &    16Y_3^2Y_4^2  &0&0 \\
0 & -8Y_3Y_4^2 &  8x_1^2\cdot Y_1Y_3Y_4^2 &8x_1^2\cdot Y_2Y_3Y_4^2 & 0\\
-8Y_3^2Y_4 & -4x_2 \cdot Y_2Y_3Y_4 &   G[4,3] & 4x_1^2x_2 \cdot Y_2^2Y_3Y_4 &0\\
8Y_3^2Y_5 & x_2 \cdot Y_2Y_3Y_5+2Y_4^2Y_5 & G[5,3] & G[5,4]&4Y_3Y_4\\
 \end{array}\right),$$ where $$G[4,3] = 4x_1^2x_2 \cdot Y_1Y_2Y_3Y_4+2x_2 \cdot Y_3^3Y_4,$$$$G[5,3] = -4x_1^2x_2 \cdot Y_1Y_2Y_3Y_5-2x_2 \cdot Y_3^3Y_5-2x_1^2 \cdot Y_1Y_4^2Y_5 \textnormal{ and }$$ $$G[5,4] = -4x_1^2x_2 \cdot Y_2^2Y_3Y_5-2x_1^2 \cdot Y_2Y_4^2Y_5$$
 and $$\begin{array}{cl}s = & a_1^8a_2^2x_2^{12}-2a_1^5a_2^4x_2^{13}+a_1^2a_2^6x_2^{14}-2a_1^6a_2^3x_2^{12}+2a_1^3a_2^5x_2^{13}-a_1^4a_2^4x_2^{12}+2a_1a_2^6x_2^{13}+\\&2a_1^8a_2x_2^{10}-4a_1^5a_2^3x_2^{11}+4a_1^2a_2^5x_2^{12}
 -4a_1^6a_2^2x_2^{10}+4a_1^3a_2^4x_2^{11}+a_2^6x_2^{12}+2a_1a_2^5x_2^{11}+\\&a_1^8x_2^8-2a_1^5a_2^2x_2^9+3a_1^2a_2^4x_2^{10}-2a_1^6a_2x_2^8+2a_1^3a_2^3x_2^9+a_1^4a_2^2x_2^8
 -2a_1^4a_2x_2^6+
 \\&2a_1a_2^3x_2^7+2a_1^2a_2^2x_2^6+2a_2^3x_2^6-2a_1^4x_2^4+2a_1a_2^2x_2^5+
 2a_1^2a_2x_2^4+1
\end{array}$$

 Using the definition of $h$ in Equation \ref{def of h}, we get that $$\left(\begin{array}{c}h_1
 \\h_2\\h_3 \\h_4 \\ h_5\end{array}\right) = s\cdot \left(\begin{array}{c}Y_1 - y_1'
 \\Y_2 - y_2'\\Y_3 - y_3' \\Y_4 - y_4' \\ Y_5 - y_5'\end{array}\right) - x_2^4 G(y') \cdot \left(\begin{array}{c}T_1
 \\T_2\\T_3 \\T_4 \\ T_5\end{array}\right),$$ where $$\begin{array}{ccl}
 y_1' &=& -18a_1^2a_2^5x_2^{12}+5a_1^3a_2^4x_2^{11}+7a_2^6x_2^{12}-18a_1a_2^5x_2^{11}+15a_1^2a_2^4x_2^{10}
 -4a_1^3a_2^3x_2^9-\\&&6a_2^5x_2^{10}+15a_1a_2^4x_2^9-12a_1^2a_2^3x_2^8+3a_1^3a_2^2x_2^7+5a_2^4x_2^8
 -12a_1a_2^3x_2^7+9a_1^2a_2^2x_2^6-\\&&2a_1^3a_2x_2^5-4a_2^3x_2^6+9a_1a_2^2x_2^5-6a_1^2a_2x_2^4+
 a_1^3x_2^3+3a_2^2x_2^4-6a_1a_2x_2^3+3a_1^2x_2^2-\\&&2a_2x_2^2+3a_1x_2+1\\
 y_2' &=& a_1^{12}x_2^{12}+2a_1^{10}a_2x_2^{12}-a_1^{11}x_2^{11}+a_1^8a_2^2x_2^{12}-2a_1^9a_2x_2^{11}+a_1^{10}x_2^{10}-a_1^7a_2^2x_2^{11}+\\&&2a_1^8a_2x_2^{10}-a_1^9x_2^9+a_1^6a_2^2x_2^{10}-2a_1^7a_2x_2^9+a_1^8x_2^8-a_1^5a_2^2x_2^9+2a_1^6a_2x_2^8-a_1^7x_2^7+\\&&a_1^4a_2^2x_2^8-2a_1^5a_2x_2^7+a_1^6x_2^6-a_1^3a_2^2x_2^7+2a_1^4a_2x_2^6-a_1^5x_2^5+a_1^2a_2^2x_2^6-2a_1^3a_2x_2^5+\\&&a_1^4x_2^4-a_1a_2^2x_2^5+2a_1^2a_2x_2^4-a_1^3x_2^3+a_2^2x_2^4-2a_1a_2x_2^3+a_1^2x_2^2+2a_2x_2^2-a_1x_2+\\&&1\\
 y_3' &=& a_1x_2^2+x_2\\
 y_4' &=& a_2x_2^3+x_2\\
 y_5' &=& \fracs{a_2^2}{4}x_2^4-\fracs{a_1^3+a_1a_2}{4}x_2^3+\fracs{a_1^2-a_2}{4}x_2^2-\fracs{a_1}{4}x_2+\fracs{1}{4}.
 \end{array}$$

 However, the output is too big to be printed.
 Following the idea in the above examples, we compute $Q$ and $g$. This is even bigger than $h$ so we print  the numerators and denominators of the coefficients just till degree 10 in the $x_i$'s. However in some cases  even the terms till degree 10 will be too many to write down and hence we will print just the first terms and ``$\ldots$" . 

 As a small remark, $Q_3$ contains also terms in degree 3 in the $T_i$ but the numerator of the coefficients have power greater than 10 and therefore they do not appear in our shortcutting.

  $$\arraycolsep=1.4pt\def\arraystretch{2.2}\begin{array}{crc}
 Q_1 =&  \fracs{\ti 3a_1^2x_1^2x_2^8-2a_1x_1^2x_2^7+4a_2x_1^2x_2^8+x_1^2x_2^6}{\ti  4a_1a_2^2x_2^5+8a_1a_2x_2^3+4a_1x_2+4a_2^2x_2^4+8a_2x_2^2+4} \cdot  T_1T_4  - \fracs{\ti x_2^6}{\ti 4a_2^2x_2^4+8a_2x_2^2+4} \cdot & T_1^2\\& + \fracs{\ti -a_1^4x_2^{10}+a_1^3x_2^9-2a_1^2a_2x_2^{10}-a_1^2x_2^8+2a_1a_2x_2^9+a_1x_2^7-a_2^2x_2^{10}-2a_2x_2^8-x_2^6}{\ti 4a_1a_2^2x_2^5+8a_1a_2x_2^3+4a_1x_2+4a_2^2x_2^4+8a_2x_2^2+4} \cdot & T_1T_2 \\& +\fracs{\ti -5a_1^4x_2^{10}+4a_1^3x_2^9-12a_1^2a_2x_2^{10}-3a_1^2x_2^8+8a_1a_2x_2^9+2a_1x_2^7-6a_2^2x_2^{10}
 -4a_2x_2^8-x_2^6}{\ti 16a_1^2a_2^2x_2^6+32a_1^2a_2x_2^4+16a_1^2x_2^2+32a_1a_2^2x_2^5+64a_1a_2x_2^3+32a_1x_2
 +16a_2^2x_2^4+32a_2x_2^2+16} \cdot & T_2^2 \\& +\fracs{\ti a_1^2x_1^2x_2^8+2a_1^2x_2^{10}+2a_1x_1^2x_2^7+4a_1x_2^9+x_1^2x_2^6+2x_2^8}{\ti 4a_1a_2^2x_2^5+8a_1a_2x_2^3+4a_1x_2+4a_2^2x_2^4+8a_2x_2^2+4} \cdot & T_1T_3 \\& + \fracs{\ti a_1x_1^2x_2^7-2a_1x_2^9+2a_2x_1^2x_2^8-4a_2x_2^{10}+x_1^2x_2^6-2x_2^8}{\ti 8a_1^2a_2^2x_2^6+16a_1^2a_2x_2^4+8a_1^2x_2^2+16a_1a_2^2x_2^5+32a_1a_2x_2^3+16a_1x_2+8a_2^2x_2^4
 +16a_2x_2^2+8}\cdot & T_2T_3 \\& + \fracs{\ti -x_1^4x_2^6+4x_1^2x_2^8-4x_2^{10}}{\ti 16a_1^2a_2^2x_2^6+32a_1^2a_2x_2^4+16a_1^2x_2^2+32a_1a_2^2x_2^5+64a_1a_2x_2^3+
 32a_1x_2+16a_2^2x_2^4+32a_2x_2^2+16} \cdot & T_3^2 \\&  + \fracs{\ti 6a_1^2x_1^2x_2^8-3a_1x_1^2x_2^7+6a_2x_1^2x_2^8+x_1^2x_2^6}{\ti  8a_1^2a_2^2x_2^6+16a_1^2a_2x_2^4+8a_1^2x_2^2+16a_1a_2^2x_2^5+32a_1a_2x_2^3+
 16a_1x_2+8a_2^2x_2^4+16a_2x_2^2+8} \cdot & T_2T_4 \\& + \fracs{\ti -x_1^4x_2^6+2x_1^2x_2^8}{\ti 8a_1^2a_2^2x_2^6+16a_1^2a_2x_2^4+8a_1^2x_2^2+16a_1a_2^2x_2^5+32a_1a_2x_2^3+
 16a_1x_2+8a_2^2x_2^4+16a_2x_2^2+8} \cdot & T_3T_4 \\& +\fracs{\ti -x_1^4x_2^6}{\ti 16a_1^2a_2^2x_2^6+32a_1^2a_2x_2^4+16a_1^2x_2^2+32a_1a_2^2x_2^5+64a_1a_2x_2^3
 +32a_1x_2+16a_2^2x_2^4+32a_2x_2^2+16}\cdot & T_4^2 \\

 Q_2 =&  \fracs{\ti -x_2^6}{\ti 4a_1^2x_2^2+8a_1x_2+4} \cdot T_2^2
 + \fracs{\ti 3a_1^2x_1^2x_2^8+3a_1x_1^2x_2^7-2a_2x_1^2x_2^8+x_1^2x_2^6}{\ti 2a_1^2x´_2^2+4a_1x_2+2} \cdot  T_2T_3
 + \fracs{\ti -x_1^4x_2^6}{\ti 4a_1^2x_2^2+8a_1x_2+4} \cdot &T_3^2 \\&
 + \fracs{\ti a_1^2x_1^2x_2^8-a_1x_1^2x_2^7+2a_2x_1^2x_2^8+x_1^2x_2^6}{\ti 2a_1^2x_2^2+4a_1x_2+2} \cdot T_2T_4
 + \fracs{\ti -x_1^4x_2^6+2x_1^2x_2^8}{\ti 2a_1^2x_2^2+4a_1x_2+2} \cdot T_3T_4
 + \fracs{\ti -x_1^4x_2^6}{\ti 4a_1^2x_2^2+8a_1x_2+4} \cdot &T_4^2 \\

 Q_3 =&
 \fracs{\ti 2a_1^3x_2^9-4a_1^2a_2x_2^{10}-2a_1^2x_2^8+2a_1a_2x_2^9+2a_1x_2^7-2a_2x_2^8-2x_2^6}{\ti 4a_1a_2^3x_2^7+12a_1a_2^2x_2^5+12a_1a_2x_2^3+4a_1x_2+4a_2^3x_2^6+12a_2^2x_2^4+12a_2x_2^2+4} \cdot &
T_1T_2 \\& + \fracs{\ti 7a_1^3x_2^9-28a_1^2a_2x_2^{10}-7a_1^2x_2^8+21a_1a_2x_2^9+7a_1x_2^7-21a_2^2x_2^{10}-21a_2x_2^8-7x_2^6}{\ti \ldots +
48a_1^2x_2^2+48a_1a_2^3x_2^7+144a_1a_2^2x_2^5+144a_1a_2x_2^3+48a_1x_2+16a_2^3x_2^6+4
8a_2^2x_2^4+48a_2x_2^2+16} \cdot & T_2^2 \\& +
\fracs{\ti 2a_1^2x_1^2x_2^8+2a_1^2x_2^{10}+4a_1x_1^2x_2^7+4a_1x_2^9-2a_2x_1^2x_2^8-2a_2x_2^{10}+2x_1^2x_2^6+2x_2^8}{\ti 4a_1a_2^3x_2^7+12a_1a_2^2x_2^5+12a_1a_2x_2^3+4a_1x_2+4a_2^3x_2^6+12a_2^2x_2^4+12a_2x_2^2+4 } \cdot & T_1T_3 \\& +
\fracs{\ti 7a_1^2x_1^2x_2^8+4a_1^2x_2^{10}+14a_1x_1^2x_2^7+8a_1x_2^9+7a_2x_1^2x_2^8+4a_2x_2^{10}+7x_1^2x_2^6+4x_2^8 }{\ti \ldots +24a_1^2x_2^2+24a_1a_2^3x_2^7+72a_1a_2^2x_2^5+72a_1a_2x_2^3+24a_1x_2+8a_2^3x_2^6+24a_2^2x_2^4+24a_2x_2^2+8 } \cdot & T_2T_3 \\& +
\fracs{\ti -7x_1^4x_2^6-8x_1^2x_2^8-4x_2^{10}}{\ti \ldots + 48a_1^2x_2^2+48a_1a_2^3x_2^7+144a_1a_2^2x_2^5+144a_1a_2x_2^3+48a_1x_2+16a_2^3x_2^6+48a_2^2x_2^4+48a_2x_2^2+16} \cdot & T_3^2 \\& +
\fracs{\ti 6a_1^2x_1^2x_2^8-4a_1x_1^2x_2^7+6a_2x_1^2x_2^8+2x_1^2x_2^6}{\ti 4a_1a_2^3x_2^7+12a_1a_2^2x_2^5+12a_1a_2x_2^3+4a_1x_2+4a_2^3x_2^6+12a_2^2x_2^4+12a_2x_2^2+4} \cdot & T_1T_4 \\& +
\fracs{\ti 21a_1^2x_1^2x_2^8-14a_1x_1^2x_2^7+35a_2x_1^2x_2^8+7x_1^2x_2^6}{\ti \ldots + 24a_1^2x_2^2+24a_1a_2^3x_2^7+72a_1a_2^2x_2^5+72a_1a_2x_2^3+24a_1x_2+8a_2^3x_2^6+24a_2^2x_2^4+24a_2x_2^2+8 } \cdot & T_2T_4 \\& +
\fracs{\ti -7x_1^4x_2^6-4x_1^2x_2^8}{\ti \ldots +  24a_1^2x_2^2+24a_1a_2^3x_2^7+72a_1a_2^2x_2^5+72a_1a_2x_2^3+24a_1x_2+8a_2^3x_2^6+24a_2^2x_2^4+24a_2x_2^2+8} \cdot & T_3T_4 \\
\end{array}$$

$$\arraycolsep=1.4pt\def\arraystretch{2.2}
\begin{array}{crc}
\phantom{Q_3 = }& +\fracs{\ti -7x_1^4x_2^6}{\ti \ldots  + 48a_1^2x_2^2+48a_1a_2^3x_2^7+144a_1a_2^2x_2^5+144a_1a_2x_2^3+48a_1x_2+16a_2^3x_2^6+48a_2^2x_2^4+48a_2x_2^2+16} \cdot & T_4^2 \\& +
\fracs{\ti a_1^4x_2^{10}-a_1^2a_2x_2^{10}-a_2^2x_2^{10}+a_2x_2^8-x_2^6}{\ti 4a_2^3x_2^6+12a_2^2x_2^4+12a_2x_2^2+4} \cdot T_1^2 +\fracs{\ti -x_2^6}{\ti 2a_2^2x_2^4+4a_2x_2^2+2} \cdot & T_1T_5  \\& + \fracs{\ti -3a_2^2x_2^{10}-6a_2x_2^8-3x_2^6}{\ti 4a_1^2a_2^2x_2^6+8a_1^2a_2x_2^4+4a_1^2x_2^2+8a_1a_2^2x_2^5+16a_1a_2x_2^3+8a_1x_2+4a_2^2x_2^4+8a_2x_2^2+4 } \cdot & T_2 T_5 \\& + \fracs{\ti 9a_1^2x_1^2x_2^8+6a_1^2x_2^{10}+9a_1x_1^2x_2^7+6a_1x_2^9+3x_1^2x_2^6+2x_2^8 }{\ti 4a_1^2a_2^2x_2^6+8a_1^2a_2x_2^4+4a_1^2x_2^2+8a_1a_2^2x_2^5+16a_1a_2x_2^3+8a_1x_2+4a_2^2x_2^4+8a_2x_2^2+4} \cdot & T_3T_5 \\& + \fracs{\ti 3a_1^2x_1^2x_2^8-3a_1x_1^2x_2^7+12a_2x_1^2x_2^8+3x_1^2x_2^6}{\ti 4a_1^2a_2^2x_2^6+8a_1^2a_2x_2^4+4a_1^2x_2^2+8a_1a_2^2x_2^5+16a_1a_2x_2^3+8a_1x_2+4a_2^2x_2^4+8a_2x_2^2+4} \cdot & T_4T_5

 \\
 \end{array}$$
 Having $Q_i$ we obtain $g_i$: 
 $$\begin{array}{rcrc}
 g_1  =  Q_1  + &{\big( \ti \ldots -6a_1^4x_2^4+30a_1^3a_2^3x_2^9+
 45a_1^2a_2^4x_2^{10}+6a_1^2a_2^2x_2^6+6a_1^2a_2x_2^4+6a_1a_2^3x_2^7+6a_1a_2^2x_2^5+6a_2^3x_2^6+1 \big)} & \cdot  T_1 \\
 
 g_2  =  Q_2 +& {\big(\ti \ldots -6a_1^4x_2^4+30a_1^3a_2^3x_2^9+45a_1^2a_2^4x_2^{10}+6a_1^2a_2^2x_ 2^6
 +6a_1^2a_2x_2^4+6a_1a_2^3x_2^7+6a_1a_2^2x_2^5+6a_2^3x_2^6+1 \big)} & \cdot  T_2 \\
 +&{\big( \ti \ldots +6a_1^6x_2^4-30a_1^5a_2^3x_2^9-45a_1^4a_2^4x_2^{10}-
 6a_1^4a_2^2x_2^6-6a_1^4a_2x_2^4-6a_1^3a_2^3x_2^7-6a_1^3a_2^2x_2^5-6a_1^2a_2^3x_2^6-a_1^2\big)}& \\
 
 g_3 = Q_3 + & {\big(\ti \ldots   -6a_1^4x_2^4+30a_1^3a_2^3x_2^9+45a_1^2a_2^4x_2^{10}+6a_1^2a_2^2x_2^6+6a_1^2a_2x_2^4+6a_1a_2^3x_2^7+6a_1a_2^2x_2^5+6a_2^3x_2^6+1 \big)}& \cdot T_3 \\ + &{\big( \ti \ldots +  24a_1^2a_2^5x_2^{10}+18a_1^2a_2^4x_2^8+a_1^2a_2^2x_2^4+a_1^2a_2x_2^2+12a_1a_2^5x_2^9+a_1a_2^3x_2^5+a_1a_2^2x_2^3+6a_2^6x_2^{10}+a_2^3x_2^4\big)}.&
 \\
 \end{array}$$

 The General Neron Desingularization is a localization of $D[Y,T]/(h,g)$.
 For this example we will need a function \begin{verbatim}invp(poly p, int bound,string param,string variab)\end{verbatim} which computes computes the inverse of \verb"p" till order \verb"bound" in $\mathbb Q($\verb"param"$)[$\verb"variab"$]$.

 The input for this example is the following:
 \begin{verbatim}
ring All = 0,(a1,a2,x1,x2,Y1,Y2,Y3,Y4),dp;
int nra = 2;
int nrx = 2;
int nry = 4;
ideal xid = x1^2-x2^3;
ideal yid = Y3^2-x1^2*Y1*Y2,Y4^2-x2*Y2*Y3;
ideal aid = 0;
poly y1,y2,y3,y4;
y3 = 1+a1*x2;
y4 = 1+a2*x2^2;
string as,xs;
if(nra != 0)
{
  as = string(var(1));
  for( int i=2;i<=nra;i++)
  {
    as = as+","+string(var(i));
  }
}
if(nrx!=0)
{
  xs = string(var(nra+1));
  for(int i=nra+2;i<=nra+nrx;i++)
  {
    xs = xs+","+string(var(i));
  }
}
y1 = y3^3*invp(y4^2,12,as,xs);
y2 = y4^2*invp(y3,12,as,xs);
y3 = x2*y3;
y4 = x2*y4;
ideal f = y1,y2,y3,y4;
desingularization(All, nra,nrx,nry,xid,yid,aid,f,"injective","debug");
 \end{verbatim}
 \end{Example}

\begin{Remark} Our algorithm works mainly for local domains of dimension one. If $A'$ is not a domain but a Cohen-Macaulay ring of dimension one then we can  build an algorithm in the idea of the proof of Theorem \ref{gr}. In this case it is necessary to change $B$ by an Elkik's trick  \cite{El} (see \cite[Lemma 3.4]{P0}, \cite[Proposition 4.6]{S}, \cite[Corollary 5.10]{P2}). The algorithm and as well Theorem \ref{gr} might be also build when $A'$ is not Cohen-Macaulay substituting in the proofs $d$ by a certain power $d^r$ such that $(0:_Ad^r)=(0:_Ad^{r+1})$. Such algorithm could be too complicated to work really.

On the other hand, if we restrict our present algorithm to the case when $A'$ is the completion  of $A$ then we might get a faster algorithm using the idea of the proof of Theorem \ref{gr}. This algorithm could be useful in the arc frame.
\end{Remark}


\begin{thebibliography}{99}

\bibitem{An} M.\ Andre, {\em Cinq exposes sur la desingularisation}, Handwritten manuscript Ecole Polytechnique Federale de Lausanne, (1991).

\bibitem{A} M.\ Artin, {\em Algebraic approximation of structures over complete local rings}, Publ. Math. IHES, {\bf 36}, (1969), 23-58.
\bibitem{Sing} W.\ Decker, G.-M.\ Greuel, G.\ Pfister, H.\ Sch{\"o}nemann: \newblock {\sc Singular} {3-1-6} --- {A} computer algebra system for polynomial computations.\newblock {http://www.singular.uni-kl.de} (2012).
\bibitem{El} R.\ Elkik, {\em Solutions d'equations a coefficients dans un anneaux henselien}, Ann. Sci. Ecole Normale Sup.,  {\bf 6} (1973), 553-604.
\bibitem{Gr} M.\ Greenberg, {\em Rational points in henselian discrete valuation rings}, Publ. Math. IHES, {\bf 31}, (1966), 59-64.

\bibitem{G} A. \ Grothedieck, J.\ Dieudonne, {\em Elements de geometrie algebrique, IV, Part 1}, Publ. Math. IHES, 1966.
\bibitem{N} A.\ Neron, {\em Modeles minimaux des varietes abeliennes sur les corps locaux et globaux}, Publ. Math.  IHES, {\bf 21}, 1964.
\bibitem{P0} D.\ Popescu, {\em General Neron Desingularization}, Nagoya Math. J., {\bf 100} (1985), 97-126.

 \bibitem{P} D.\ Popescu, {\em General Neron Desingularization and approximation}, Nagoya Math. J., {\bf 104}, (1986), 85-115.
 \bibitem{P1} D.\ Popescu, {\em Letter to the Editor. General Neron Desingularization and approximation}, Nagoya Math. J., {\bf 118} (1990), 45-53.
\bibitem{P2} D.\ Popescu, {\em Artin Approximation}, in "Handbook of Algebra", vol. 2, Ed. M. Hazewinkel, Elsevier, 2000, 321-355.
\bibitem{Sp} M.\ Spivakovski, {\em A new proof of D. Popescu's theorem on smoothing of ring homomorphisms}, J. Amer. Math. Soc., {\bf 294} (1999), 381-444.
\bibitem{S} R.\ Swan, {\em Neron-Popescu desingularization}, in "Algebra and Geometry", Ed. M. Kang, International Press, Cambridge, (1998), 135-192.

\end{thebibliography}
\end{document}